\documentclass[12pt]{amsart}
\usepackage{amssymb,amsmath,amscd,graphicx,
latexsym,amsthm}
\usepackage{amssymb,latexsym,eufrak,amsmath,amscd,graphicx}
  \usepackage[all]{xy}
  \usepackage{pdfsync}
\theoremstyle{plain}
\newtheorem{theorem}{Theorem}
\newtheorem{proposition}[theorem]{Proposition}

\newtheorem{lemma}[theorem]{Lemma}

\newtheorem{proposition.definition}[theorem]{Proposition/Definition}

\newtheorem{theoremalpha}{Theorem}

\theoremstyle{definition}

\newtheorem{remark}[theorem]{Remark}

\newtheorem{conjecture}[theorem]{Conjecture}

\newcommand{\lra}{\longrightarrow}

\newcommand{\noi}{\noindent}
\newcommand{\PP}{\mathbf{P}}

\newcommand{\CC}{\mathbf{C}}

\newcommand{\OO}{\mathcal{O}}

\newcommand{\FF}{\mathcal{F}}

\newcommand{\bull}{_{\bullet}}

\newcommand{\frakm}{\mathfrak{m}}

\newcommand{\HH}[3]{H^{{#1}} \big( {#2} , {#3}
\big) }

\newcommand{\pr}{\prime}

\newcommand{\Div}{\text{Div}}

\newcommand{\Jac}{\textnormal{Jac}}

\newcommand{\Linser}[1]{| \mspace{1.5mu} {#1}
\mspace{1.5mu} |}
\newcommand{\linser}[1]{\Linser{  {#1}  }}

\newcommand{\gon}{\textnormal{gon}}
\newcommand{\pro}{{pr}}

\newcommand{\Sym}{\textnormal{Sym}}

\newcommand{\ev}{\textnormal{ev}}

\numberwithin{equation}{section}
\numberwithin{theorem}{section}

\begin{document}

\title[Gonality Conjecture]{The gonality conjecture on  syzygies of algebraic curves of large degree}

 \author{Lawrence Ein}
  \address{Department of Mathematics, University Illinois at Chicago, 851 South Morgan St., Chicago, IL  60607}
 \email{{\tt ein@uic.edu}}
 \thanks{Research of the first author partially supported by NSF grant DMS-1001336.}

 \author{Robert Lazarsfeld}
  \address{Department of Mathematics, Stony Brook University, Stony Brook, New York 11794}
 \email{{\tt robert.lazarsfeld@stonybrook.edu}}
 \thanks{Research of the second author partially supported by NSF grant DMS-1439285.}
 
\maketitle

 \section*{Introduction}

 The purpose of this note is to show that a small variant of the methods used by Voisin   in \cite{Voisin1} and \cite{Voisin2}   leads to a surprisingly quick proof of the gonality conjecture of \cite{GL3}, asserting that one can read off the gonality of an algebraic curve $C$ from its syzygies in the embedding defined by any one line bundle of sufficiently large degree.  More generally, we establish a necessary and sufficient condition for the asymptotic vanishing of the weight one syzygies of the module associated to an arbitrary line bundle on $C$.

 Let $C$ be a smooth complex projective curve of genus $g \ge 2$,   and let $L$ be a very ample line bundle of degree $d$ on $C$ defining an embedding 
 \[  C \ \subseteq \ \PP H^0(C,L) \ = \ \PP^r.  \]
 Starting with the work of Green in \cite{Kosz1}, \cite{Kosz2} there has been a great deal of interest in  understanding connections between the geometry of $C$ and $L$ and their syzygies. More precisely, 
   write $S = \Sym\,  \textnormal{H}^0(C,L)$ for the homogeneous coordinate ring of $\PP^r$, and denote by   \[ R=  R(L)= \oplus_m H^0(C, mL)\]  the graded $S$-module associated to $L$. Consider next the minimal graded free resolution     $E\bull = E\bull(L)$ of R over $S$:
 \[
 \xymatrix{
0 \ar[r] & E_{r-1} \ar[r]& \ldots \ar[r] &    E_2 \ar[r] & E_1 \ar[r]  & E_0 \ar[r] & R \ar[r] & 0 , }\]
where $ E_p =\oplus S(-a_{p,j})$. 
Note that if $L$ is normally generated  then $E_0 = S$, in which case $E\bull$ gives rise to a minimal resolution of the homogeneous ideal $I = I_{C/\PP^r}$ of $C$ in $\PP^r$. As customary, we denote    by $K_{p,q}(C;L)$  the vector space of minimal generators of $E_p$ in degree $p+q$, so that
\[    E_p \ = \ \bigoplus_q\,   K_{p,q}(C;L) \otimes_{\CC}  S(-p-q). \]
We will be concerned here with investigating the grading of $E\bull(L)$ -- ie determining which of the $K_{p,q}$ are non-vanishing -- when $L$ has very large degree.

It is elementary that if $H^1(C,L) = 0$  then $K_{p,q}(C;L) = 0$ for $q \ge 3$. Moreover, work of Green \cite{Kosz1} and others shows that if $d = \deg(L) \gg 0$, so that in particular  $ r = d-g$, then:
\begin{align*} K_{p,0} (C;L) \, \ne \, 0 \ &\Longleftrightarrow \ p =0; \\
K_{p,2}(C;L) \, \ne \, 0 \ &\Longleftrightarrow \ r-g \, \le p \, \le r - 1. \end{align*}
It follows from this that
\[  K_{p,1}(C; L) \, \ne \, 0 \ \text{  for } 1 \, \le \, p \, \le r-1-g, \] 
 but these results leave open the question of when $K_{p,1}(C;L) \ne 0$ for $p \in [r-g, r-1]$. 
 Our first main result is that this is determined by the gonality $\gon(C)$ of $C$, ie the least degree of a branched covering $C \rightarrow \PP^1$. 
 \begin{theoremalpha} \label{Gonality.Thm}
 If $\deg(L) \gg 0$, then 
 \[  K_{p,1}(C; L) \, \ne \, 0  \ \Longleftrightarrow \ 1 \, \le \, p \, \le \, r - \gon(C).\]
 \end{theoremalpha}
 Thus one can read off the gonality of a curve from the resolution of the ideal of $C$ in the embedding defined by any one line bundle of sufficiently large degree. The cases $p =r-1, p = r-2$ were established by Green  \cite{Kosz1}, and the general statement    was conjectured  in \cite{GL3}, where it was observed that if $1 \le p \le r -  \gon(C)$, then $K_{p,1}(C;L) \ne 0$.\footnote{In fact, suppose that $p: C \rightarrow \PP^1$ is a branched covering of degree $k$. Then when $\deg(L) \gg 0$ the linear spaces spanned by the fibres of $p$ sweep out a $k$-dimensional scroll $S \subset \PP^r$ containing $C$. But the resolution of $I_{S/\PP^r}$ has a linear strand of length $r-k$, which forces $K_{p,1}(C;L) \ne 0$ for $1 \le p \le r-k$. Thus the essential content of the Theorem is that if $K_{r-k,1}(C;L) \ne 0$ and $\deg L \gg 0$,  then $C$ carries a pencil of degree $\le k$.} Using Voisin's results \cite{Voisin1}, \cite{Voisin2}  on syzygies of general canonical curves, Aprodu and Voisin \cite{Aprodu}, \cite{AproduVoisin} proved the statement of the Theorem   for a general curve of each gonality. We show (Remark \ref{Effective.Gonality}) that the conclusion of the Theorem holds  for instance once $\deg(L)  \ge  g^3 $, but we suspect that it should be enough to assume a lower bound on $d$ that is linear in $g$.

 Theorem \ref{Gonality.Thm} follows from a more general result concerning the  weight one asymptotic syzygies associated to an arbitrary divisor $B$. Specifically, fix a line bundle $B$ on $C$, and with $L$ as above  consider  the $S = \Sym\, H^0(L)$ module 
 \[  R \ = \, R(B;L) \ = \ \bigoplus_m H^0(C, B + mL). \]
One can again form the graded minimal free resolution $E\bull(B;L)$ of $R(B;L)$ over $S$, giving rise to Koszul cohomology groups
$K_{p,q}(C,B;L)$. As in the case $B = \OO_C$ discussed in the previous paragraphs, the $K_{p,0}$ and the $K_{p,2}$ are completely controlled when $\deg L \gg 0$, and so the issue is  to understand the weight one groups $K_{p,1}(C, B;L)$ when $L$ has large degree. 

Recall  that $B$ is said to be $p$-\textit{very ample} if every effective divisor $\xi$ of degree $(p+1)$ on $C$ imposes independent conditions on the sections of $B$, i.e. if the natural map
\[   H^0(C,B) \lra H^0(C, B \otimes \OO_{\xi}) \]
is surjective for every $\xi \in C_{p+1} =_{\text{def}} \Sym^{p+1}C$.
Our second main result is:
\begin{theoremalpha} \label{Kp1(B).Theorem}
Fix $B$ and $p \ge 0$. Then 
\[  K_{p,1}(C, B; L ) \, = \,0  \ \text{ for all $L$ with $\deg L \gg 0$} \]
if and only if $B$ is $p$-very ample. 
\end{theoremalpha}
\noi Serre duality implies that the vector spaces
\[   
K_{p,q}(C, B;L) \ \ \text{ and } \ \ K_{r-1-p, 2-q}(C, K_C - B; L)
\]
are naturally dual, $K_C$ being the canonical divisor of $C$, and one then finds that Theorem A is equivalent to the case $B = K_C$ of Theorem B. While this is arguably the most interesting instance of the result, it will become clear that decoupling $B$ and $L$ is helpful in guiding the argument.

When $B$ fails to be $p$-very ample, it is natural to introduce the invariant
\[
\gamma_p(B) \ = \ \dim \big\{ \xi \in C_{p+1}\,  \big | \, H^0(B) \lra H^0(B \otimes \OO_\xi )\text{ not surjective } \big\}.
\]
\begin{theoremalpha} \label{Hilb.Poly.Thm}
Let $L_d = dA + E$, where $A$ is an ample line bundle on $C$ and $E$ is arbitrary. Fix  $B$ and $p$, and assume that $B$ is not  $p$-very ample. Then  there is a polynomial $P(d)$ of degree $\gamma_p(B)$ in $d$ such that
\[  
\dim K_{p,1}(C, B; L_d) \, = \, P(d) \ \ \text{ for } d \gg 0. \]
 \end{theoremalpha}
 \noi In some cases, we are also able to compute the leading coefficient of $P(d)$. We note that  Yang \cite{Yang} has recently proven (by somewhat related arguments) that the dimensions of the vector spaces $K_{p,0}$ and $K_{p,1}$ grow polynomially on an arbitrary variety.
 
Theorems B and C follow in a surprisingly simple manner from a small variant of the Hilbert scheme computations pioneered by Voisin in her proof \cite{Voisin1}, \cite{Voisin2} of Green's conjecture for general canonical curves.  It is well known that $K_{p,1}(C,B;L)$ can be computed as the cohomology of the Koszul-type complex
\[
\Lambda^{p+1}H^0(L) \otimes H^0(B) \lra \Lambda^p H^0(L) \otimes H^0(B+L) \lra \Lambda^{p-1}H^0(L) \otimes H^0(B+2L), 
\]
and the basic strategy is to realize this complex geometrically. 
In brief, a line bundle $B$ on $C$ determines a vector bundle $E_B = E_{p+1, B}$ of rank $p+1$ on the symmetric product $C_{p+1}$ whose fibre at a point $\xi \in C_{p+1}$ is the vector space $H^0(C, B \otimes \OO_\xi)$. The natural map $H^0(B) \lra H^0(B \otimes \OO_\xi)$ globalizes to a homomorphism of vector bundles
\[  \ev_B = \ev_{p+1, B}:  H^0(C,B) \otimes_{\CC} \OO_{C_{p+1}} \lra \ E_B, \tag{*}
\]
and evidently $\ev_B$ is surjective as a map of vector bundles if and only if $B$ is $p$-very ample. On the other hand, if $N_L = \det E_L$, then it is well-known that 
$H^0(N_L) = \Lambda^{p+1} H^0(C, L)$, and twisting (*) by $N_L$ gives rise to a vector bundle map
\[  H^0(C, B) \otimes N_L \lra E_B \otimes N_L. \tag{**}\]
Computations of Voisin identify $H^0(C_{p+1}, E_B \otimes N_L)$ with the space $Z_{p,1}(C, B; L)$ of Koszul cycles, and hence $K_{p,1}(C,B;L) = 0$ if and only if the homomorphism
\[  H^0(C, B) \otimes H^0(C_{p+1},N_L) \lra H^0(C_{p+1}, E_B \otimes N_L)\]
determined by (**) is surjective. But assuming that $B$ is $p$-very ample, so that (**) is surjective as a map of bundles, this follows for $\deg L \gg 0$ simply by applying Serre-Fujita  vanishing to the kernel of (**). We note that the main difference from Voisin's  set-up  -- apart from separating $B$ and $L$, which clarifies the issue -- is that we push down to the symmetric product rather than working on the universal family over it. Some related computations had earlier appeared  in the paper \cite{Laz1}, where it was shown that one could see the syzygies of canonical curves in  cohomology related to the cotangent bundle $E_{\Omega_C}$ of the symmetric product,  but it has to be admitted that nothing came of these. 

We are grateful to Marian Aprodu, Gabi Farkas, B. Purnaprajna, Frank Schreyer, David Stapleton, Bernd Sturmfels, Brooke Ullery and  Claire Voisin for valuable discussions and encouragement.


\section{Proofs}

This section is devoted to the proofs of Theorems A, B and C from the Introduction. We keep the notation introduced there.\footnote{In addition, we continue to allow ourselves to be a little sloppy in confounding additive and multiplicative notation for divisors and line bundles.} Thus $C$ is a smooth projective curve of genus $g$, and $L$ is a very ample line bundle of degree $d$  on $C$ defining an embedding
\[  C \ \subseteq \ \PP H^0(L) \, = \, \PP^r. \] We fix an arbitrary line bundle on $B$ on $C$, and we are intrested in the Koszul cohomology groups \[ K_{p,q}(B;L)\ =\ K_{p,q}(C, B;L)\]arising as the cohomology of the Koszul-type complex:

\Small
\vskip -16pt
\[
\Lambda^{p+1}H^0(L) \otimes H^0(B+(q-1)L) \lra \Lambda^p H^0(L) \otimes H^0(B+qL) \lra \Lambda^{p-1}H^0(L) \otimes H^0(B+(q+1)L). 
\]
\normalsize
We recall that results of Green and others  imply that if $d = \deg(L) \gg 0$, then $K_{p,q}(B;L) = 0$ for all $q \ge 3$, and:
\begin{align*} 
K_{p,0}(B;L) \, \ne \, 0 \ &\Longleftrightarrow  \ p \in [0, h^0(B)-1] \\
K_{p,2}(B;L) \, \ne \, 0 \ &\Longleftrightarrow \  p \in [r - h^1(B), r - 1]  \end{align*}
(cf \cite[Proposition 5.1, Corollary 5.2]{ASAV}).\footnote{In particular, if  $H^0(B) = 0$ then $K_{p,0}(B;L) = 0$ for all $p$, and if $H^1(B) = 0$, then $K_{p,2}(B;L) = 0$ for all $p$ provided that $\deg L \gg 0$.} So the issue is to understand which of the groups $K_{p,1}(B;L)$ vanish when $\deg L \gg 0$.

Write $C_k $ for the $k^{\text{th}}$ symmetric product of $C$, viewed as parameterizing all effective divisors on $C$ of degree $k$. We consider the  commutative diagram:
\begin{equation}
\begin{gathered}
 \xymatrix{& C & \\
 C \times C_p \   \ar@{^{(}->}[rr] ^{j_{p+1}}\ar[dr]_{ \sigma_{p+1}} \ar[ur]^{pr_1}& & \  C  \times C_{p+1} \ar[dl]^{pr_2}  \ar[ul]_{pr_1}\\ & C_{p+1}&
  } 
 \end{gathered}
 \end{equation}
where $\sigma_{p+1}$ and $ j_{p+1}$ are the maps defined by 
\[   \sigma_{p+1}( x , \xi) \, = \, x + \xi \ \ , \ \ j_{p+1}(x, \xi) \, = \, (x , x + \xi). \]
Note that $\sigma_{p+1}$ realizes $C \times C_p$ as the universal family of degree $p+1$ divisors over $C_{p+1}$. 

The proofs revolve around two well-studied tautological sheaves on $C_{p+1}$. First given a line bundle $B$ on $C$, define
\[   E_B \ = \ E_{p+1, B} \ =_{\text{def}} \ \sigma_{p+1, *} \,  \pro_1^* (B).\]
Thus $E_B$ is a vector bundle of rank $p+1$ on $C_{p+1}$ whose fibre at $\xi \in C_{p+1}$ is identified with the vector space $H^0(C, B \otimes \OO_\xi)$.  It follows from the construction that  $H^0(C_{p+1}, E_B) = H^0(C, B)$, which gives rise to a homomorphism: \begin{equation} \label{ev.B.equation}
\ev_B \ = \ \ev_{p+1,B} : H^0(C, B) \otimes_{\CC} \OO_{C_{p+1}}\lra E_B
\end{equation}
of vector bundles on $C_{p+1}$.  Evidently $\ev_B$ is surjective if and only if $B$ is $p$-very ample. Next, given a line bundle $L$ on $C$, put
\[    N_L \ = \ N_{p+1,L} \ = \ \det E_L. \]
Note that $\Lambda^{p+1} \ev_L$ determines  a map
\[   \Lambda^{p+1} H^0(C,L) \lra H^0(C_{p+1}, N_L), \]
and it was established eg in \cite{EGL} and \cite{Voisin1} that this is an isomorphism. Twisting $\ev_B$ by $N_L$, one arrives at the vector bundle map
\begin{equation} \label{Basic.VB.Map}
H^0(C,B) \otimes_{\CC} N_L \lra E_B \otimes N_L 
\end{equation}
that lies at the heart of the proof. 

Our main results follow immediately from two lemmas whose proofs appear at the end of this section. The first, which is effectively due to Voisin, states that $K_{p,1}(B;L) = 0$ if and only if \eqref{Basic.VB.Map}
is surjective on global sections. The second asserts that as $L$ gets very positive on $C$,  the corresponding line bundles $N_L$  become sufficiently positive on $C_{p+1}$ to satisfy a Serre-type vanishing theorem.
\begin{lemma}[Voisin] \label{Voisin.Lemma}
The global sections of $E_B \otimes N_L$ are identified with the space 
\[
Z_{p,1}(B;L) \ = \ \ker \Big ( \Lambda^p H^0(L) \otimes H^0(B+L) \lra  \Lambda^{p-1}H^0(L)\otimes H^0(B+ 2L)\Big )\]
of Koszul cycles, and the homomorphism 
\[ 
H^0(C,B) \otimes H^0(C_{p+1}, N_L) \, = \, H^0(C,B) \otimes \Lambda^{p+1} H^0(C,L) \lra H^0(C_{p+1}, E_B \otimes N_L) \]
arising from \eqref{Basic.VB.Map} is identified with the Koszul differential. In particular, 
\[  K_{p,1}(C, B; L) \ = \ 0 \] if and only if the bundle map \eqref{Basic.VB.Map} determines a surjection on global sections. 
\end{lemma}

\begin{lemma} \label{Serre.Vanishing.Lemma} Let $\FF$ be any coherent sheaf on $C_{p+1}$. There exists an integer $d_0 = d_0(\FF)$ having the property that if $d = \deg(L) \ge d_0(\FF)$, then
\[
\HH{i}{C_{p+1}}{\FF \otimes N_L} \ = \ 0 \ \ \text{ for }\ i > 0.
\]
\end{lemma}

Granting the lemmas for now, we prove the main results.

\begin{proof} [Proof of Theorem \ref{Kp1(B).Theorem}]  Assume that $B$ is $p$-very ample, so that   $\ev_B$  in \eqref{ev.B.equation} is surjective. 
Denote by $M_B = M_{p+1,B}$ its kernel:
\begin{equation} \label{MBundle.Eqn}
0  \lra M_B \lra H^0(C,B) \otimes \OO_{C_{p+1}} \lra E_B \lra 0. \end{equation}
To show that $K_{p,1}(B;L) = 0$ when $\deg L \gg 0$, it suffices by Lemma \ref{Voisin.Lemma} to prove that
\begin{equation} \label{Vanishing.H1} \HH{1}{C_{p+1}}{M_B \otimes N_L} \ = \ 0 \end{equation}
for very positive $L$. But this follows from Lemma \ref{Serre.Vanishing.Lemma}. Conversely, if $\ev_B$ is not surjective, then it is elementary -- and we will see momentarily in the proof of Theorem \ref{Hilb.Poly.Thm} -- that $K_{p,1}(B;L) \ne 0$ for every sufficiently positive $L$. 
\end{proof}
\begin{remark} Proposition \ref{Eff.Bound.Proposition} below gives an effective lower bound on $\deg(L)$ that is sufficient to guarantee the vanishing \eqref{Vanishing.H1}. \qed
\end{remark}

\begin{proof}[Proof of Theorem \ref{Hilb.Poly.Thm}] Denote by $M_B$ and $\FF_B$ respectively the kernel and cokernel of $\ev_B$:
\begin{equation} \label{ThmCEqn}
 0 \lra M_B \lra H^0(B) \otimes \OO_{C_{p+1}} \lra E_B \lra \FF_B \lra 0. \end{equation}
Taking $L_d = dA + E$ as in the statement of the Theorem, put $N_d = N_{L_d}$.  We will see in the proof of Lemma \ref{Serre.Vanishing.Lemma} below that 
\[ N_d  \ = \ N_E + dS_A, \]
where $S_A$ is an ample divisor on $C_{p+1}$. On the other hand, it follows from the two lemmas that for $d \gg 0$
\[    K_{p,1}(C, B; L_d) \ = \   \HH{0}{C_{p+1}}{\FF_B \otimes N_d}. \]
Therefore  $\dim K_{p,1}(B;L_d)$ is given for $d \gg 0$ by the Hilbert polynomial of $\FF_B \otimes N_E$ with respect to $S_A$. But   $\gamma_p(B) = \dim \textnormal{Supp}\,  \FF_B$, and the result follows. 
\end{proof}

\begin{remark}
This argument shows that $ K_{p,0}(C,B;L_d)= \HH{0}{C_{p+1}}{M_B \otimes N_d}$ provided that $d$ is large. Hence (assuming that $p \le r(B)$) the dimension of this  Koszul group always grows as a polynomial of degree $(r(B) - p)$ in $d$ when $d \gg 0$.\footnote{The arguments of   \cite{Yang} show that  analogously on a variety of  dimension $n$,  $\dim K_{p,0}$ grows as a polynomial of degree $n(r(B) - p)$.} In other words, it is the growth of the $K_{p,1}$ groups that exhibit interesting dependence on geometry. \qed \end{remark}

We next recall the well-known argument that the case $B = K_C$ of Theorem  \ref{Kp1(B).Theorem} implies the Gonality Conjecture.
\begin{proof} [Proof of Theorem \ref{Gonality.Thm}]
Fix $p \le g$. We  need to show that if $\deg(L) \gg 0$, and if
\[  K_{r(L) - p,1}(C; L) \ \ne \ 0 \tag{*}, \]
then $C$ carries a pencil of degree $\le p$. By duality, (*) implies that 
\[  K_{p-1, 1}(C, K_C; L) \ \ne  \ 0, \]
and hence by Theorem \ref{Kp1(B).Theorem} there exists an effective divisor $\xi \in C_p$ of degree $p$ that fails to impose independent conditions on $|K_C|$. But then $\xi$ moves in a non-trivial linear series thanks to Riemann-Roch.
\end{proof}

We conclude this section by proving the two lemmas stated above. 
\begin{proof} [Proof of Lemma \ref{Voisin.Lemma}]
It follows from the projection formula and the constructions that
\begin{align*} \HH{0}{ C_{p+1}}{ E_B \otimes N_L} \ &= \ \HH{0}{C \times C_p}{pr_1^*B \otimes \sigma_{p+1}^*N_L)}\\ &= \ \HH{0}{C \times C_p}{(j_{p+1})^*(pr_1^*B \otimes pr_2^*N_L))}. \end{align*} 
Moreover the map induced by \eqref{Basic.VB.Map} on global sections is identified with  the restriction 
\[
\HH{0}{C \times C_{p+1}}{B \boxtimes N_L } \lra \HH{0}{C \times C_p} {{(B \boxtimes N_L )}|(C\times C_p)}.\]
But this is exactly Voisin's Hilbert-schematic interpretation of Koszul cohomology, and from this point one can argue just as in \cite[Lemma 5.4]{AproduNagel}. In brief, one observes that on $C \times C_{p}$ one has an isomorphism
\[    j_{p+1}^* \big( N_{p+1,L} \big) \ = \ \big( L \boxtimes  N_{p,L} \big)(-D),  \]
where  $D\subseteq C \times C_p$ is the image of $j_p: C \times C_{p-1} \hookrightarrow C \times C_p$. Therefore
\[
\HH{0}{C\times C_p}{(j_{p+1})^*(B \boxtimes N_{p+1,L}) }\]
is identified with 
\[ \ker \Big( \, \HH{0}{C \times C_p}{\OO_C(B+ L) \boxtimes N_{p,L}}  \lra  \HH{0}{C \times C_{p-1}}{\OO_C(B+2L) \boxtimes N_{p-1,L}}\,  \Big), \]
and the assertion follows.
\end{proof}

\begin{proof} [Proof of Lemma \ref{Serre.Vanishing.Lemma}]
Given a divisor $A$ on $C$, the divisor $T_A =_{\text{def}} \sum pr_i^*(A)$ on the Cartesian product $C^{ p+1}$ descends to a divisor $S_A = S_{p+1,A}$ on $C_{p+1}$. For example, if $A = x_1 + \ldots + x_d$, then 
\[  S_A \ = \ C_{p, x_1} + \ldots + C_{p, x_d} \ \in \ \Div(C_{p+1}) , \]
where $C_{p,x}$ denotes the image of the map $C_p \hookrightarrow C_{p+1}$ given by $\xi \mapsto \xi  +x$. One has $S_{A_1 + A_2} = S_{A_1} + S_{A_2}$, and $S_A$ is ample on $C_{p+1}$ if and only if $A$ is ample on $C$. Observe next that if $L$ is line bundle on $C$, then
$
N_{L+A}= N_L + S_A $ on $C_{p+1}$. This is well-known, but it can be checked directly from the definitions by observing that if $x \in C$ is a point  then there is an exact sequence
\[  0 \lra E_L \lra E_{L(x)} \lra \OO_{C_{p,x}} \lra 0 \]
of sheaves on $C_{p+1}$. 

Now fix an ample divisor $A$ of degree $a$ on $C$ and a coherent sheaf $\FF$ on $C_{p+1}$. By Fujita-Serre vanishing, there exists an integer $m_0 = m_0(\FF)$ such that if $P$ is any nef divisor on $C_{p+1}$, then 
\[  \HH{i}{C_{p+1}}{\FF(m  S_A + P )} \ = \ 0 \ \text{ for } \, i > 0 \tag{*} \]
whenever $m \ge m_0$. Put
\[  d_0 \ = \ d_0(\FF) \ = \ (2g + p) + m_0a, \]
and suppose that $\deg(L) \ge d_0$. Then $L = L_0 + m_0A$ where $L_0$ is $p$-very ample, and in particular $N_{L_0}$ is globally generated. Therefore 
\[  N_L \ = \  m_0 S_A + \textnormal{ ( nef ) }, \]
and so (*) gives the required vanishing. 
\end{proof}

\section{Complements}

This section is devoted to some additional results, and a conjecture about what one might hope for in higher dimensions.

We start by establishing an effective version of Theorem \ref{Kp1(B).Theorem}. Since the statement is presumably far from optimal we only sketch the proof. \begin{proposition} \label{Eff.Bound.Proposition}
Assume that $B$ is $p$-very ample. Then
$K_{p,1}(C,B;L) =0$ for every line bundle $L$ with
\begin{equation} \label{Effective.Bound}
\deg(L) \ > \ (p^2 + p + 2)(g-1)   \, + \, (p+1)\deg(B). \end{equation}
\end{proposition} 
\begin{proof} [Sketch of Proof] Keeping notation as in the proof of Theorem \ref{Kp1(B).Theorem}, one needs to prove that $\HH{1}{C_{p+1}}{M_B \otimes N_L} =0$  when $\deg(L)$ satisfies the stated bound. If $h^0(C, B) > 2(p+1)$, we replace $H^0(C, B)$ in \eqref{MBundle.Eqn} by a general subspace of dimension $2p+2$ to define a vector bundle $M_B^\pr$ of rank $p+1$ sitting in an exact sequence
\[  0 \lra M_B^\pr \lra M_B \lra \oplus\, {\OO}_{C_{p+1}}\lra 0, \]
and one is reduced to proving that $\HH{1}{C_{p+1}}{M_B^\pr \otimes N_L} =0$. Note that $M^\pr \otimes N_B$ is globally generated and that $\det M_B^\pr = -N_B$. 

We assert that if $L$ satisfies \eqref{Effective.Bound}, then
\[
N_L - (p+1)N_B - K_{C_{p+1}} \  \text{ is ample} . \tag{*}\]
Granting this, we see that if \eqref{Effective.Bound} holds, then
\[  M_B^\pr\otimes N_L \ = \ \big( M_B^\pr \otimes N_B) \otimes \det(M_B^\pr \otimes N_B) \otimes K_{C_{p+1}} \otimes A \]
where is $A$ is ample, so the Griffiths vanishing theorem \cite[7.3.2]{PAG} applies. For (*), it is equivalent to check the statement after pulling back by the quotient $\pi : C^{p+1} \rightarrow C_{p+1}$. One has $\pi^* N_L= T_L - \Delta$, where $T_L = \sum pr_i^* L$ is the symmetrization of $L$ and  $\Delta \in \Div( C^{p+1})$ is the union of the pairwise diagonals. Since $K_{C_{p+1}} = N_{K_C}$, the claim (*) reduces with some computation to the fact that if $D$ is a divisor on $C$, then $T_D + \Delta$ is nef on $C^{p+1}$ if and only if $\deg D \ge p(g-1)$. \end{proof}

\begin{remark} \label{Effective.Gonality} The Proposition guarantees that we can detect whether $K_C$ is $p$-very ample (or equivalently, whether $\gon(C) \ge p+2$) by the vanishing of $K_{p,1}(C, K_C;L)$ for any $L$ with 
\[ 
\deg(L) \ > \ (p^2 + 3p +3)(g-1). 
\]
But in any event $\gon(C) \le \tfrac{g+3}{2}$, and it follows (with some computation) that  the gonality of $C$ is detemined by the weight one syzygies of $C$ with respect to any line bundle of degree $\ge g^3$.  However we expect that such cubic bounds are far from optimal: one hopes that it is enough that the degree of $L$ grows linearly in $g$.  \qed \end{remark}

As suggested by Schreyer, we observe next that in some cases one can use the proof of Theorem \ref{Hilb.Poly.Thm}
 to get more information about the polynomial $P(d)$ appearing there. We focus on the most interesting case $B =K_C$, and content ourselves  with illustrating the method in a  simple instance. Specifically, suppose that $C$ carries finitely many pencils 
\[
\alpha_1, \ldots, \alpha_s \ \in \ W^1_{p+1}(C)
\]
of degree $p+1$, while  no other divisors of degree $p+1$ on $C$ move in non-trivial linear series.  We assume also that each $\alpha_i$ is (scheme-theoretically) an isolated point in $W^1_{p+1}(C)$ in the sense that the multiplication maps
\[
H^0(\alpha_i) \otimes H^0(K_C - \alpha_i) \lra H^0(K) \tag{*}
\]
are surjective for each $i$.\footnote{Recall that the Gieseker-Petri theorem  asserts that the hypothesis holds automatically for a general curve of genus $g = 2p$, in which case $s$ is given by a certain Catalan number.}
\begin{proposition} \label{Enumerative.Prop}
Under the hypotheses just stated, take $L_d = d \cdot x$ for some point $x \in C$. Then for $d \gg 0$,
\[  \dim K_{p,1}(C, K_C; L_d) \ = \ s \cdot d + \textnormal{ ( constant ) }. \]
\end{proposition}
\noi We note that O'Dorney and Yang \cite{OY} have made some interesting computations of the dimensions of  $ K_{p,0}(C, K_C; L_d)$ on a general curve, including determining the leading coefficient of the resulting polynomial. \begin{proof}[Sketch of Proof of Proposition \ref{Enumerative.Prop}]
Note that each $\alpha_i$ determines a copy of $\PP^1 = \linser{\alpha_i}$ sitting in the symmetric product $C_{p+1}$,
and these are precisely the positive-dimensional fibres of the Abel-Jacobi map \[ u = u_{p+1} : C_{p+1} \lra \Jac^{p+1}(C).\] Now when $B = K_C$, the evaluation  \eqref{ev.B.equation} is identified with the coderivative $du$ of $u$, and by a well-known computation \cite[Chapt. IV.4]{ACGH}, the condition (*) implies that
\[   \textnormal{coker}  \, du \ = \ \oplus_{i=1}^s  \,\Omega^1_{\linser{\alpha_i}}. \]
 In particular, the sheaf $\FF_{K_C}$ appearing in \eqref{ThmCEqn} has rank one along each $\PP^1 = \linser{\alpha_i}$. On the other hand, if $L_d = d \cdot x$ then the divisor $N_d$ has degree $d + \textnormal{ (constant) }$ along $\linser{\alpha_i}$, so each of these  copies of $\PP^1$ contributes a term of the same shape to the Hilbert polynomial of $\FF_{K_C}$. \end{proof}

Finally, we make some remarks about what one might expect in higher dimensions. Let $X$ be a smooth projective variety of dimension $n$, and let $L_d = dA + E$ where $A$ is an ample and $E$ an arbitrary divisor on $X$. Given a line bundle $B$ on $X$, one would like to give geometric conditions on $B$ in order that 
\begin{equation} \label{Kp1.Van.Higher.Dim} K_{p,1}(X, B; L_d)  \ = \ 0  \ \ \text{for all } d \gg 0:  \end{equation}
as explained above and in \cite[Problem 7.2]{ASAV} this is the most interesting group from an asymptotic viewpoint. It is conceivable that it suffices to assume that $B$ is $p$-very ample in the sense that $H^0(B)$ imposes independent conditions on every subscheme $\xi \subseteq X$ of length $p+1$, but this seems out of reach. On the other hand, recall that $B$ is said to be $p$-\textit{jet very ample} if for every effective zero-cycle
\[  z \ = \ a_1 x_1 + \ldots + a_s x_s \]
of degree $p+1$ on $X$, the natural mapping
\[   \HH{0}{X}{B}  \lra \HH{0}{X}{B \otimes  \OO_X/  \frakm_1^{a_1} \cdot \ldots \cdot \frakm_s^{a_s} } \]
is surjective, where $\frakm_i  \subseteq \OO_X$ is the ideal sheaf of $x_i$.  When $\dim X = 1$ this is the same as $p$-very ample, but in higher dimensions the condition on jets is stronger. \begin{conjecture}
If $B$ is $p$-jet very ample, then \eqref{Kp1.Van.Higher.Dim} holds.
\end{conjecture}
\noi It is very possible that  the ideas of \cite{Yang} will be helpful for this.

 %
 %
 %
 %

 \end{document}